\documentclass{amse-new}

\numberwithin{equation}{section} %%% Equations numbered by section. If you don't want it, please delete it.

\newcommand{\ee}{\mathbb{E}}

\newcommand{\rr}{\mathbb{R}}
\newcommand{\pp}{\mathbb{P}}
\newcommand{\qq}{\mathbb{Q}}

\def\HH{\mathbb H}

\def\VV{\mathbb V}
\def\Lip{{\mathrm{{\rm Lip}}}}

\def\HS{\rm HS}

\begin{document}

 \PageNum{1}
 \Volume{201x}{Sep.}{x}{x}
 \OnlineTime{August 15, 201x}
 \DOI{0000000000000000}
 \EditorNote{Received February 20, 2019, accepted October 23, 2019}

\abovedisplayskip 6pt plus 2pt minus 2pt \belowdisplayskip 6pt
plus 2pt minus 2pt
%%%%%%%%%%%%%%%%
\def\vsp{\vspace{1mm}}
\def\th#1{\vspace{1mm}\noindent{\bf #1}\quad}
\def\proof{\vspace{1mm}\noindent{\it Proof}\quad}
\def\no{\nonumber}
\newenvironment{prof}[1][Proof]{\noindent\textit{#1}\quad }
{\hfill $\Box$\vspace{0.7mm}}
\def\q{\quad} \def\qq{\qquad}
\allowdisplaybreaks[4]
%%%%%%%%%%%%%%%%%%%%%%%%%%%%%%%%%%%%%%%%%%%%%%%%%%%%%%%%%%%%%%%%%%%%%%%%%%%%%%%%%%%%%%%%%%%%%%%
%%-------------------       Beginning of  Author's Definitions       -------------------%%
%%                     Note: You may add your own definitions here.

%%-------------------         the end of  Author's Definitions           -------------------%%

\AuthorMark{Ma Y. T.  and Wang R.}                             %%%  appear on the head of even pages  %%%

\TitleMark{Transport inequality of SPDE  with L\'evy noise}  %%% Running Title, appear on the head of odd pages  %%%

\title{Transportation cost inequalities for stochastic reaction-diffusion equations  with L\'evy noises  and non-Lipschitz reaction terms        %%%   Main Title of your paper  %%%
\footnote{The first author is supported by National  Natural Science Foundation of China (11571043, 11431014, 11871008); the second author is supported by National  Natural Science Foundation of China (11871382, 11671076). }}                  %%%   the Fund which you are supported by  %%%

\author{Yu Tao \uppercase{MA}}             %%%  1st Author's information   %%%
    { School of Mathematical Sciences $\&$ Lab. Math. Com. Sys., Beijing Normal University, Beijing 100875, P. R. China\\
    E-mail\,$:$ mayt@bnu.edu.cn }

\author{Ran \uppercase{WANG}$^{1)}$\footnote{1) Corresponding author} }     %%%  2nd Author's info, if exists, or you may delete this part directly  %%%
    { School of Mathematics and Statistics, Wuhan University, Wuhan 430072, P. R. China\\
    E-mail\,$:$  rwang@whu.edu.cn }

\maketitle%

\Abstract{ For  stochastic reaction-diffusion equations  with
L\'evy noises  and non-Lipschitz reaction terms, we prove that $W_1H$ transportation cost inequalities hold for their
invariant probability measures and for their process-level laws on the path space with respect to the $L^1$-metric. The proofs are based on the Galerkin approximations.}      % the abstract

\Keywords{ Stochastic reaction-diffusion equation;  Poisson random measure;
Transportation cost inequality}        % the keywords

\MRSubClass{28A35, 60E15}      % MR(2000) Subject Classification

\section{Introduction}

  Let $(E, d)$ be a metric space equipped with $\sigma$-field $\mathcal B$ such that $d(\cdot,
\cdot)$ is $\mathcal{B}\times\mathcal{B}$-measurable. Given $p\ge 1$
and two probability measures $\mu$ and $\nu$ on $E,$  define the $L^p$-Wasserstein distance between
$\mu$ and $\nu$:
$$W_{p, d}(\mu, \nu)=\inf\bigg(\int\int_{E\times E} d(x, y)^p\pi(dx, dy)\bigg)^{1/p},$$
where the infimum is taken over all probability measures $\pi$ on
the product space $E\times E$ with marginal distribution $\mu$ and
$\nu$. The  relative entropy of $\nu$ with respect
to $\mu$ is defined by \begin{equation}\label{kullback} {\mathbf H}(\nu|\mu)=\begin{cases}
\int_E\log\frac{d\nu}{d\mu}d\nu, & \text{if } \nu<\!\!<\mu; \\
+\infty, & \text{otherwise}.
\end{cases}
\end{equation}

We say that the probability $\mu$ satisfies a $W_pH$ transportation cost-information inequality on $(E, d)$ if there exists  a constant
$C>0$ such that for any probability measure $\nu\in \mathcal M_1(E)$ (the space of all probability measures on $E$),
 \begin{equation}\label{WpH}
W_{p, d} (\mu, \nu)\le \sqrt{ 2C{\mathbf H}(\nu|\mu)}.
  \end{equation}

Let $\alpha:[0,\infty)\rightarrow[0,\infty]$ be a non-decreasing    left-continuous convex function, with $f(0)=0$.     $\mu$ is said to satisfy the $\alpha-W_pH$  if for all probability measure $\nu$ on $E$,
 \begin{equation}\label{W1H0}
\alpha\big(W_{p, d}(\mu, \nu)\big)\le {\mathbf H}(\nu|\mu).
 \end{equation}
  The inequality \eqref{WpH} is a particular case of \eqref{W1H0}
with $\alpha(t)=t^{2/p}/(2C)$ for any $t\ge0$.

The properties $W_pH, p=1,2$ are of particular interest. They  have been brought into relation with the phenomenon of measure concentration, functional inequalities, Hamilton-Jacobi's equation, optimal transport problem,  large deviations, see, e.g.,  \cite{BGL2001, BGL, BG,  DGW, FS, Goz, GL, Ledoux, Tal, Villani,  WangFY} and references therein. For example, we give   Gozlan-L\'eonard's characterization for $\alpha-W_1H$ transportation cost inequality.

\begin{theorem}\label{GL} (Gozlan-L\'eonard \cite{GL}) Let $\alpha: [0,\infty)\to
[0, \infty]$ be a non-decreasing left continuous convex function
with $\alpha(0)=0.$ The following properties are equivalent:

$(i).$ The $\alpha-W_1H$ inequality below holds
$$\alpha(W_{1, d}(\nu, \mu))\le H(\nu|\mu), \quad \forall \nu\in \mathcal M_1(E);$$

$(ii).$ For every $f: (E, d)\to\rr$ bounded and Lipschitzian with
$\|f\|_{\Lip}\le 1,$ \begin{equation}\label{equi} \int_E
e^{\lambda(f-\mu(f))}d\mu\le e^{\alpha^{*}(\lambda)}, \quad
\lambda>0, \end{equation} where $\alpha^*(\lambda):=\sup_{r\ge
0}(r\lambda-\alpha(r))$ is the semi-Legendre transformation;

$(iii).$ Let $(\xi_k)_{k\ge 1}$ be a sequence of independent and identically distributed  random variables  taking
values in $E$ of common law $\mu$.  For every $f: E\to\rr$ with
$\|f\|_{\Lip}\le 1,$
$$\pp\bigg(\frac{1}{n}\sum_{k=1}^nf(\xi_k)-\mu(f)>r\bigg)\le e^{-n\alpha(r)}, \; r>0, \; n\ge 1. $$
 \end{theorem}
The equivalence of (i) and (ii) is a generalization of
Bobkov-G\"{o}tze's criterion \cite{BG} for quadratic $\alpha$, and (iii) gives a probability meaning to the $\alpha-W_1H$ inequality.

The $W_2H$ inequalities  on the path spaces of stochastic (partial) differential equations  driven by Gaussian noises have been investigated by many authors, for example,  see \cite{BWY, DGW, FU, P} for stochastic differential equations (SDEs) and  \cite{BH, SZ, WZ2006} for stochastic partial differential equations (SPDEs).

The $\alpha-W_1H$ inequalities  on the path spaces of SDEs  with jumps have also been investigated, see   \cite{Wunew}   for SDEs driven by pure jump processes,   \cite{Ma} for SDEs driven by both Gaussian and jump
noises,  and  \cite{SY} for  regime-switching diffusion processes.

The transportation inequalities for non-globally dissipative SDEs with jumps were studied  in M. Majka \cite{Majka}, by using the mirror coupling for the jump part and the reflection coupling   for the Brownian part, for  bounding Malliavin derivatives of solutions of SDEs with both jump and Gaussian noise. We would also like to mention the works of \cite{LW} and  \cite{Majka2017} for  the   exponential convergence with respect to the $L^1$-Wasserstein distance when the drift  is dissipative outside a compact set.

The  aim of this paper is to prove that  the $\alpha-W_1H$ transportation cost inequalities hold for stochastic reaction-diffusion equations driven by both Gaussian and L\'evy noises under the $L^1$ distance in the path space. The reaction term can be chosen to be Lipschitz continuous, or to be a polynomial, for example $f(x)=-x^3+C_1x$ for some $C_1\in \mathbb R$.  The main ingredient in our study is
  the finite dimensional approximations,  which is  more or less standard in the
literature for the Lipschitz case, but it is  difficult in  the non-Lipschitz case.

This paper is organized as follows. In Section 2, we present the framework for the stochastic reaction-diffusion equations with jumps and with Lipschitz reaction terms, and then prove the transportation cost inequalities. In Section 3,    we first establish some tightness results for approximating processes of the system with non-Lipschitz reaction terms, and then  prove  the transportation cost inequalities.

\section{Transportation cost inequalities for SPDE  with L\'evy noise  and  Lipschitz  reaction term }

\subsection{SPDE  with L\'evy noise  and  Lipschitz  reaction term }

Let $\HH: = L^2(0,1)$ be the space of square integrable real-valued functions on   $[0,1]$.
The norm and the inner product on $\HH$ are
denoted by $\|\cdot\|_{\HH}$ and $\langle\cdot,\cdot\rangle_{\HH}$, respectively. Let $\HH^k(0,1)$ be the Sobolev space of all functions in $\HH$ whose derivatives up to order $k$  also belong  to $\HH$.  $\HH^1_0(0,1)$ is the subspace of $\HH^1(0,1)$ of all functions whose values at $0$ and $1$ vanish.
Let $\Delta$  be the Laplace operator on $\HH$:
\begin{align*}
\Delta x:= \frac{\partial^2}{\partial \xi^2}x(\xi),
\quad  x\in \HH^2(0,1) \cap \HH^1_0(0,1).
\end{align*}
It is well known that $\Delta$ is the infinitesimal generator of a strongly continuous semigroup
$S(t):=e^{t\Delta}, t\geq0$.  $\{e_k(\xi):=\sqrt{2}\sin(k\pi\xi)\}_{k\geq 1}$ is an orthonormal  basis of $\HH$ consisting of the eigenvectors of $\Delta$, i.e.,
$$\Delta e_k=-\lambda_k e_k\ \ \ \ {\rm with} \ \ \lambda_k=k^2\pi^2.$$

For any $\theta\in\rr$,  let
 $$ \HH_{\theta}:=\left\{ x=\sum_{k \ge1} x_k e_k: (x_k)_{k\ge 1}\in \rr, \sum_{k\ge1} \lambda_k^{\theta} |x_k|^2<\infty\right\},$$
endowed with norm
$$
\ \|x\|_{\HH_\theta}:=\left(\sum_{k\ge1} \lambda_k^{\theta} |x_k|^2\right)^{1/2}.
$$
Then, for any $\theta>0$,  $\HH_{\theta}$ is densely and compactly embedded in $\HH$. Particularly,  denote $\VV:=\HH_{1}=\HH^1_0(0, 1)$, whose dual space is $\VV^{*}=\HH_{-1}$. The norm and the inner product on $\VV$ are
denoted by $\|\cdot\|_{\VV}$ and $\langle \cdot, \cdot\rangle_{\VV}$, respectively. If $_{\VV^*}\langle \cdot, \cdot\rangle_{\VV}$ denotes the duality between $\VV$ and its dual space $\VV^*$,  we have
$$
_{\VV^*}\langle u, v\rangle_{\VV}=\langle u, v\rangle_{\HH}, \ \ \  \text{for any } u\in \HH, v\in \VV.
$$

Let $(\Omega, \mathcal F, (\mathcal F_t)_{t\ge0}, \mathbb P)$ be a filtered probability space, $(\mathbb X, \mathcal B(\mathbb X))$ be  a measurable space, and $\vartheta$  a $\sigma$-finite measure on it.   Let $N(dt,du)$ be a Poisson random measure on $\mathbb R_+\times \mathbb X$ with intensity measure $dt\vartheta(du)$,   $\widetilde{N}(dt, du)=N(dt,du)-dt\vartheta(du)$
 the  compensated Poisson random measure,   and $(\beta^k)_{k\ge1}$  a sequence of  independent and identically distributed  one dimensional standard  Brownian motions on the probability space $(\Omega, \mathcal F, (\mathcal F_t)_{t\ge0}, \mathbb P)$. Then  $\beta_t:=\sum_{k\ge1}\beta^k_t e_k $ is  an $\HH$-cylindrical Brownian motion.

 Consider the  following  SPDE on the Hilbert space $\HH$:
\begin{eqnarray}\label{SPDE}
\left\{
 \begin{array}{lll}
 & dX_t=\Delta X_tdt+f(X_t)dt+\sigma(X_t)d\beta_t+\int_{\mathbb X} G(X_{t-},v)\widetilde{N}(dt,dv); \\
 &X_0=x\in \HH,\\
 \end{array}
\right.
\end{eqnarray}
where $x$ is   $\mathcal F_0$-measurable.
 The coefficients $f: \HH\rightarrow \HH$, $\sigma: \HH\rightarrow \mathcal L_2(\HH; \HH)$ (the space of all Hilbert-Schmidt operators from $\HH$ to $\HH$), $G:\HH\times \mathbb X\rightarrow \HH$  are   Fr\'echet continuously differentiable, and they satisfy the following conditions:

\begin{itemize}
  \item[(H1).] The reaction term  $f$ is Lipschitz  continuous, i.e., there exists a positive constant $C_{f}>0$ such that
  $$\|f(x)-f(y)\|_{\HH}\le C_{f} \|x-y\|_{\HH}, \ \ \ \forall x,y\in \HH.$$
  \item[(H2).] $\sigma$ is Lipschitz  continuous, i.e., there exists a positive constant $C_{\sigma}>0$ such that
  $$\|\sigma(x)-\sigma(y)\|_{\HS}\le C_{\sigma} \|x-y\|_{\HH}, \ \ \ \forall x,y\in \HH.$$
  \item[(H3).] $G$ satisfy the following conditions:
  \begin{align}
&\int_{\mathbb X}\|G(x, v)-G(y,v)\|_{\HH}^2\vartheta(dv)\le C_G \|x-y\|_{\HH}^2;\\
&\int_{\mathbb X}\|G(x, v)\|_{\HH}^2\vartheta(dv)\le C_G(1+\|x\|_{\HH}^2).
  \end{align}
\end{itemize}

Let $\mathbb{D}([0, T]; \mathbb H)$ be the space  of all
right continuous with left limits   $\mathbb H$-valued functions on $[0, T]$, endowed with the Skorokhod topology. Recall the following results about equation \eqref{SPDE} from \cite[Theorem 3.3]{RZ} and  \cite[Lemma 3.13]{YZZ}.

\begin{theorem}\cite{RZ, YZZ} Under Conditions (H1)-(H3), for any $x\in L^2(\Omega;\HH)$, there exists a unique $\HH$-valued progressively measurable process $\{X_t\}_{t\in[0,T]}\in \mathbb D([0,T];\HH) \cap L^2((0,T];\VV)$ for any $T>0$, and for any $\phi\in \VV$, it holds that a.s.,
\begin{align}\label{eq: solution-integ}
\langle X_t, \phi\rangle_{\HH}=&\langle x, \phi\rangle_{\HH}+\int_0^t{_{\VV^*}}\langle \Delta X_s,\phi\rangle_{\VV} ds
+\left\langle \int_0^t \sigma(X_s)d\beta_s,\phi\right\rangle_{\HH} \notag\\&+\int_0^t\int_{\mathbb X}\langle G(X_{s-}, v),\phi\rangle_{\HH} \widetilde N(ds,dv).
\end{align}
Furthermore, we have
\begin{align}\label{eq: solution-integ2}
\mathbb E\left[\sup_{0\le t\le T}\|X_t\|_{\HH}^2+ \int_0^T\|X_t\|_{\VV}^2dt\right]<\infty.
\end{align}
\end{theorem}

 \begin{remark}
 Recall that $\lambda_1$ is the first eigenvalue of $-\Delta$.
Let
\begin{equation}\label{eq L0}
K:=2\lambda_1-(2 C_{f} +C_{\sigma}^2+C_G).
\end{equation}
   By  (H1)-(H3), we know that   for all $x_1, x_2\in \HH$,
  \begin{align}\label{eq:diss}
  &2\left\langle x_1-x_2, \Delta (x_1-x_2)\right \rangle_{\HH}+2\left\langle x_1-x_2, f(x_1)-f(x_2)\right \rangle_{\HH} +\left\|\sigma(x_1)-\sigma(x_2)\right\|_{\HS}^2\notag\\
  &\ \ \ +\int_{\mathbb X}\left\|G(x_1, v)-G(x_2, v)\right\|_{\HH}^2\vartheta(dv)\notag\\
 \le& -K  \|x_1-x_2\|_{\HH}^2.
  \end{align}
When $K>0$, \eqref{eq:diss} is the globally dissipative condition, which is used to guarantee the existence of invariant measure and further to obtain the transportation cost inequality for this invariant measure. For example,  under the globally dissipative condition, the following results hold:  the Markov process $\{X_t\}_{t\ge0}$ admits an invariant measure (e.g. see \cite[Chapter 16]{PZ}); for the SPDE driven by additive Gaussian noise (i.e., $G=0$ and $\sigma$ is a constant matrix),  Da Prato {\it et al.} \cite{DDG} obtained the
log-Sobolev inequality for the invariant measure of $\{X_t\}_{t\ge0}$, and  Zhang and Wu \cite{WZ2006} obtained the log-Sobolev inequalities for  the   process-level law on the continuous path  space with respect to the $L^2$-metric, which are stronger than the transportation cost inequality $W_2H$ by \cite{OV};
 for finite dimensional  stochastic differential
equations with jumps, the  transportation cost inequalities $W_1H$ were obtained for  their
invariant probability measure as well as for their process-level law on the right
continuous paths space with respect to the $L^1$-metric, see \cite{Wunew} and \cite{Ma}.
  \end{remark}

\subsection{Transportation cost inequalities for SPDE  with L\'evy noise  and Lipschitz reaction term }

 Recall the following result, which tells that the $W_pH$-inequality is stable under the weak convergence.
\begin{lemma}\label{lem:stability:weak}\cite[Lemma 2.2]{DGW} Let $(E,d)$ be a metric, separable and complete space and $(\mu_n, \mu)_{n\in\mathbb N}$ a family of probability measures on $E$. Assume that $\mu_n\in W_pH(C)$ for all $n\in \mathbb N$ and $\mu_n\rightarrow \mu$ weakly. Then $\mu\in W_pH(C)$.

\end{lemma}

The first named author \cite{Ma} proved the  transportation cost inequalities for SDE  with L\'evy noises under the globally dissipative condition. Now, we use the finite dimensional approximation's technique and Lemma \ref{lem:stability:weak} to prove the  transportation cost inequalities for SPDE \eqref{SPDE}.

\begin{theorem}\label{main1} Assume  Conditions  (H1)-(H3) hold, $K>0$ in \eqref{eq L0},   $\|\sigma(x)\|_{\rm HS}\le
\bar\sigma$ for any $x\in \HH$ and there is some
Borel-measurable function $\bar G(u)$ on $\mathbb X$
such that $|G(x, v)|\le\bar G(v)$ for all  $x\in \HH, v\in \mathbb X$ and
\begin{equation}\label{condi} \exists \lambda>0: \quad
\Lambda(\lambda):=\int_{\mathbb X}\left(e^{\lambda \bar G(v)}-\lambda\bar G(v)-1\right)\vartheta(dv)<\infty. \end{equation}
The following properties hold
\begin{itemize}
\item[(1)]  $\{X_t\}_{t\ge 0}$ admits a unique invariant probability
measure $\mu,$ and for any $p\in [1, 2], t>0, \nu\in \mathcal M_1(\HH)$,
 \begin{equation}\label{contractive}
W_{p, d}(\nu P_t, \mu)\le e^{-K t}W_{p, d}(\nu, \mu),
 \end{equation}
  where $d(x, y)=\|x-y\|_{\HH}$.
\item[(2)] For each $T>0, x\in \HH$,  the Markov transition probability   $P_{T}(x, dy)$ satisfies the following
$\alpha-W_1H$ transportation inequality:
 \begin{equation}\label{W1H} \alpha_T(W_{1,
d}(\nu, P_T(x, dy)))\le {\bf H}(\nu| P_T(x, dy)), \quad\forall \nu\in
\mathcal M_1(\HH), \end{equation}
where
\begin{align*}\alpha_T(r):=&\sup_{\lambda>0}\left\{r\lambda-\int_{0}^{T}\Lambda\left(e^{-K t}\lambda\right)dt-\frac{\bar\sigma^2
\lambda^2}{4K}\left(1-e^{-2K T}\right)\right\}\\
\ge& \frac1{K}
\gamma^{\ast}_{1/2}(K r),
\end{align*}
with
$\gamma_a(\lambda):=\Lambda(\lambda)+a\bar\sigma^2\lambda^2/
2$ and $\gamma^{\ast}_{a}(r):=\sup_{\lambda\ge0}\big(r\lambda-\gamma_a (\lambda)\big), \, r\ge 0$. In particular, for
the invariant probability measure $\mu,$ \begin{equation}\label{spe}
\frac{1}{K}\gamma^{\ast}_{1/2} \big(K W_{1, d}(\nu,
\mu)\big)\le\alpha_{\infty}\big(W_{1, d}(\nu, \mu)\big)\le
{\mathbf H}(\nu|\mu), \end{equation}
for all $\nu\in \mathcal M_1(\HH)$.
 \item[(3)] For each $T>0$, the law $\pp_{x}$ of $X_{[0,
T]},$ the solution of \eqref{SPDE} with   $X_0=x$ being a fixed point in $\HH$, satisfies the $W_1H$ on the
space $\mathbb{D}([0, T]; \HH)$,
\begin{equation}\label{L1}\alpha_T^P\big(W_{1, d_{L^1}}(\mathbb{Q}, \pp_x)\big)\le {\mathbf H}(\mathbb{Q}|\pp_{x}), \quad \forall
\mathbb{Q}\in \mathcal M_1\big(\mathbb{D}([0, T]; \HH)\big)\end{equation} and
\begin{align}\label{L1bound}\alpha^P_T(r):=&\sup_{\lambda>0}\bigg(\lambda
r-\int_0^T\Lambda(\eta(t)\lambda)
dt-\frac{\bar \sigma^2\lambda^2}{2}\int_0^T\eta^2(t)dt\bigg)\notag\\
\ge &
T\gamma^{\ast}_1(r K/T), \end{align} where $\eta(t):=(1-e^{-Kt})/K$  and
$d_{L^1}(\gamma_1,\gamma_2)=\int_0^T\|\gamma_1(t)-\gamma_2(t)\|_{\HH}dt$  for any $\gamma_1, \gamma_2\in \mathbb D([0,T];\mathbb H)$.
\end{itemize}\end{theorem}

According to the proof of  \cite[Corollary 2.7]{Wunew}, we can apply part $(3)$ of Theorem \ref{main1} to obtain
the following result.
\begin{corollary}   In the framework of Theorem \ref{main1}, let
$\mathcal{A}$ be a  family of real Lipschitzian functions
$f$ on $\HH$ with $\|f\|_{{\rm Lip}}:=\sup_{x\neq y\in \HH}|f(x)-f(y)|/\|x-y\|_{\HH}\le1$,  and
$$Z_T:=\sup_{f\in\mathcal{A}}\bigg(\frac1T\int_0^T
f(X_s)ds-\mu(f)\bigg).$$ We have for all $r, T>0,$
$$\log\pp \big(Z_T>\ee [Z_T]+r\big)\le
-\alpha_T^P(T r)\le -T\gamma_{1}^{\ast}(Kr). $$ The same
inequality holds for $Z_T=W_{1,d}(L_T, \mu),$ where
$L_T:=\frac{1}{T}\int_0^T\delta_{X_s}ds$ is the empirical
measure. \end{corollary}

\begin{proof}[Proof of Theorem \ref{main1}]
Recall that $\{e_1, e_2,\cdots \}$ is an orthonormal basis of $\HH$. Let $P_n:\VV^*\rightarrow \HH_n$ be defined by
\begin{equation}\label{eq proj}
 P_ny:=\sum_{i=1} {_{\VV^*}}\langle y,e_i\rangle_{\VV} e_i,\ \ y\in \VV^*.
\end{equation}
Then $P_n|_{\HH}$ is also the orthogonal projection onto $\HH_n$ in $\HH$ and we have
$$
{_{\VV^*}}\langle P_n\Delta u,v\rangle_{\VV}=
 \langle P_n \Delta u,v\rangle_{\HH}=
{_{\VV^*}}\langle \Delta u,v\rangle_{\VV}, \ \ \ \text{for all  }u\in \VV, v\in \HH_n,
$$
and
$
\|v\|_{\HH_n}=\|v\|_{\HH}  \ \text{for all } v\in \HH_n.
$

Let $\beta^{(n)}_t=\sum_{i=1}^n \beta_i e_i$. Then for any $x\in \HH$, we have
$$
P_n\sigma(x)d\beta_t=P_n\sigma(x)d\beta_t^{(n)}.
$$

Consider the following Galerkin approximations: $X^{(n)}\in \HH_n$ denotes the solution of the following  stochastic differential equation:
\begin{eqnarray}\label{F_G-01}
dX_t^{(n)}&=& P_n \Delta X_t^{(n)}dt+P_nf\left(X_t^{(n)}\right)dt+P_n\sigma\left(X_t^{(n)}\right) d\beta^{(n)}_t\nonumber\\
 & &+\int_{\mathbb{X}}P_n G\left(X_{t-}^{(n)},v\right)\widetilde N(dt,dv),
\end{eqnarray}
with initial condition $X^{(n)}_0=P_nX_0=P_n x\in \mathbb H_n$.
 By the  Lipschitz continuity of  $f, \sigma$ and $G$, we know that  the  equation (\ref{F_G-01})
 admits a unique strong solution $X_n\in \mathbb D([0,T];\HH_n)\cap L^2([0,T]; \VV_n)$. Furthermore, we have
 \begin{align}\label{eq: solution-integ3}
\sup_{n\ge1}\mathbb E\left[\sup_{0\le t\le T}\left\|X_t^{(n)}\right\|_{\HH_n}^2+ \int_0^T\left\|X_t^{(n)}\right\|_{\VV_n}^2dt\right]<\infty.
\end{align}
Since  $\lambda_1$ is the first eigenvalue of $-\Delta$, by  (H1)-(H3),  we have that for any $x_1, x_2\in \HH_n$,
\begin{align}\label{eq:diss n}
  &2\left\langle x_1-x_2, P_n \Delta (x_1-x_2)\right \rangle+2\left\langle x_1-x_2, P_nf(x_1)-P_nf(x_2) \right\rangle \notag\\ &+\left\|P_n(\sigma(x_1)-\sigma(x_2))\right\|_{\HS}^2+\int_{\mathbb X}\left\|P_nG(x_1, v)-P_nG(x_2, v)\right\|_{\HH_n}^2\vartheta(dv)\notag\\
 \le &  -K\left\|x_1-x_2\right\|_{\HH_n}^2.
  \end{align}
   If $K>0$, by \cite[Theorem 2.2]{Ma}, we know that  all the results in Theorem \ref{main1} replacing $X$ by $X^{(n)}$ hold. Hence, Proposition \ref{prop convergence 1} below  together with  Lemma \ref{lem:stability:weak}, implies  that  Theorem \ref{main1} holds.
The proof is complete.
\end{proof}

\begin{proposition}\label{prop convergence 1}
Under Conditions   (H1)-(H3), for any $t\ge0$, we have
\begin{equation}\label{eq convergence 1}
\lim_{n\rightarrow \infty}\mathbb E\left[\sup_{0\le s\le t}\left\|X_s-X^{(n)}_s \right\|_{\HH}^2+\int_0^t\left\|X_s-X^{(n)}_s \right\|_{\VV}^2ds \right]=0.
\end{equation}

\end{proposition}

\begin{proof} Applying  It\^o's formula to $\left\|X_t-X^{(n)}_t\right\|_{\HH}^2$, we obtain that
\begin{align}\label{eq 1.1}
&\left\|X_t-X^{(n)}_t\right\|_{\HH}^2+2\int_0^t\left\|X_s-X^{(n)}_s\right\|_{\VV}^2ds\notag\\
=& \|(I-P_n)x\|_{\HH}^2+ 2\int_0^t \left\langle X_s-X^{(n)}_s, f(X_s)-P_nf\left( X^{(n)}_s \right) \right\rangle_{\HH} ds \notag\\
&+  2\int_0^t \left\langle X_s-X^{(n)}_s, \left[\sigma(X_s)-P_n\sigma\left( X^{(n)}_s\right) \right]d\beta_s\right\rangle_{\HH} \notag\\
&+ \int_0^t\left\|\sigma(X_s)-P_n\sigma\left( X^{(n)}_s \right)\right\|_{\HS}^2ds\notag\\
&+  2\int_0^t\int_{\mathbb X} \left\langle X_{s-}-X^{(n)}_{s-},  G\left(X_{s-}, v\right)-P_nG\left( X^{(n)}_{s-},v\right)\right\rangle_{\HH} \widetilde N(ds, dv) \notag\\
&+\int_0^t\int_{\mathbb X}\left\|G\left(X_{s-}, v\right)-P_nG\left( X^{(n)}_{s-},v\right)\right\|_{\HH}^2N(ds,dv).
\end{align}
Taking the supremum up to $t$ in \eqref{eq 1.1}, and then taking the expectation, we have

\begin{align}\label{eq 1.2}
&\ee \left[\sup_{0\le s\le t}\left\|X_s-X^{(n)}_s\right\|_{\HH}^2\right]+2\ee \int_0^t\left\|X_s-X^{(n)}_s\right\|_{\VV}^2ds\notag\\
\le &\ee\left[\|(I-P_n)x\|_{\HH}^2\right]+ 2\ee \int_0^t \left|\left\langle X_s-X^{(n)}_s, f(X_s)-P_nf\left( X^{(n)}_s \right) \right\rangle_{\HH}\right| ds \notag\\
&+  2\ee \left[\sup_{0\le s\le t}\left|\int_0^s \left\langle X_r-X^{(n)}_r, \left[\sigma(X_r)-P_n\sigma\left( X^{(n)}_r\right) \right]d\beta_r\right\rangle_{\HH}\right| \right]\notag\\
&+ \ee \int_0^t\left\|\sigma(X_s)-P_n\sigma\left( X^{(n)}_s \right)\right\|_{\HS}^2ds\notag\\
&+  2\ee \left[\sup_{0\le s\le t} \left|\int_0^s\int_{\mathbb X} \left\langle X_{r-}-X^{(n)}_{r-},  G\left(X_{r-}, v\right)-P_nG\left( X^{(n)}_{r-},v\right)\right\rangle_{\HH} \widetilde N(dr, dv)\right|\right] \notag\\
&+\ee \left[\sup_{0\le s\le t} \int_0^s\int_{\mathbb X}\left\|G\left(X_{r-}, v\right)-P_nG\left( X^{(n)}_{r-},v\right)\right\|_{\HH}^2N(dr,dv)\right]\notag\\
=:&\ee\left[\|(I-P_n)x\|_{\HH}^2\right]+I_1^{(n)}(t)+I_2^{(n)}(t)+\cdots+ I_5^{(n)}(t).
\end{align}
For $I_1^{(n)}$, by the Lipschitz continuity of $f$ and the elementary inequality $2ab\le a^2+b^2$ for all $a,b>0$,  we have
\begin{align}\label{eq I1}
I_1^{(n)}(t)\le & 2\ee \int_0^t \left|\left\langle X_s-X^{(n)}_s, f(X_s)-P_nf\left(X_s \right) \right\rangle_{\HH}\right| ds\notag\\
&+ 2\ee \int_0^t \left|\left\langle X_s-X^{(n)}_s,  P_n\left(f(X_s)-f\left( X^{(n)}_s \right)\right) \right\rangle_{\HH}\right|ds \notag\\
\le & \left(1+ 2 C_f \right)\ee \int_0^t\left\|X_s-X^{(n)}_s\right\|_{\HH}^2ds+\ee \int_0^t \left\|(I-P_n)f(X_s)\right\|_{\HH}^2ds.
\end{align}
By Burkholder-Davis-Gundy's inequality  and the Lipschitz continuity of $\sigma$, we have
\begin{align}\label{eq I2}
I_2^{(n)}(t)\le &  2\ee \left[\sup_{0\le s\le t}\left|\int_0^s \left\langle X_r-X^{(n)}_r, \left[\sigma(X_r)-P_n\sigma\left(X_r\right) \right]d\beta_r\right\rangle_{\HH}\right| \right]\notag\\
&+2\ee \left[\sup_{0\le s\le t}\left|\int_0^s \left\langle X_r-X^{(n)}_r, P_n\left[\sigma(X_r)-\sigma\left( X^{(n)}_r\right) \right]d\beta_r\right\rangle_{\HH}\right| \right]\notag\\
\le & 4 \ee \left[ \int_0^t \left\|X_s-X^{(n)}_s\right\|_{\HH}^2\cdot \left\|(I-P_n)\sigma\left(X_s\right) \right\|_{\HS}^2ds \right]^{\frac12}+4C_{\sigma} \ee \left[ \int_0^t \left\|X_s-X^{(n)}_s\right\|_{\HH}^4 ds \right]^{\frac12}\notag\\
\le& 2 \ee  \int_0^t \left\|X_s-X^{(n)}_s\right\|_{\HH}^2ds+ 2 \ee  \int_0^t \left\|(I-P_n)\sigma\left(X_s\right) \right\|_{\HS}^2ds  \notag\\
&+ 4C_{\sigma} \ee \left[\sup_{0\le s\le t}\left\|X_s-X^{(n)}_s\right\|_{\HH} \cdot \left(\int_0^t\left\|X_s-X^{(n)}_s\right\|_{\HH}^2 ds \right)^{\frac12}\right]\notag\\
\le& 2 \ee  \int_0^t \left\|X_s-X^{(n)}_s\right\|_{\HH}^2ds+ 2 \ee  \int_0^t \left\|(I-P_n)\sigma\left(X_s\right) \right\|_{\HS}^2ds  \notag\\
&+ \frac14  \ee \left[\sup_{0\le s\le t}\left\|X_s-X^{(n)}_s\right\|_{\HH}^2\right] +16C_{\sigma}^2 \ee  \int_0^t \left\|X_s-X^{(n)}_s\right\|_{\HH}^2 ds.
\end{align}
For the third term $I_3^{(n)}$, by the Lipschitz continuity of $\sigma$,  we have
\begin{align}\label{eq I3}
I_3^{(n)}(t)\le &  \ee \int_0^t\left\|(I-P_n)\sigma(X_s)\right\|_{\HS}^2ds\notag\\
&+  \ee \int_0^t\left\|P_n\sigma(X_s)-P_n\sigma\left( X^{(n)}_s \right)\right\|_{\HS}^2ds\notag\\
\le & \ee \int_0^t\left\|(I-P_n)\sigma(X_s)\right\|_{\HS}^2ds + C_{\sigma}^2 \ee \int_0^t\left\|X_s-X^{(n)}_s\right\|_{\HH}^2ds.
\end{align}
By Burkholder-Davis-Gundy's inequality and the Lipschitz continuity of $G$, we have
\begin{align}\label{eq I4}
 I_4^{(n)}(t)\le& 2\ee \left[\sup_{0\le s\le t} \left|\int_0^s\int_{\mathbb X} \left\langle X_{r-}-X^{(n)}_{r-}, (I-P_n) G\left(X_{r-}, v\right) \right\rangle_{\HH} \widetilde N(dr,dv)\right|\right]\notag\\
 & +2\ee \left[\sup_{0\le s\le t} \left|\int_0^s\int_{\mathbb X} \left\langle X_{r-}-X^{(n)}_{r-},  P_n G\left(X_{r-}, v\right)-P_nG\left( X^{(n)}_{r-},v\right)\right\rangle_{\HH} \widetilde N(dr,dv)\right|\right]\notag\\
 \le & 4\ee \left[\int_0^t\int_{\mathbb X} \left\|X_s-X_s^{(n)}\right\|_{\HH}^2\cdot \left\|(I-P_n)G(X_s,v)\right\|_{\HH}^2\vartheta (dv)ds  \right]^{\frac12} \notag\\
 &+4  \ee \left[\int_0^t\int_{\mathbb X}\left\|X_s-X^{(n)}_s\right\|_{\HH}^2\cdot \left\|P_n G\left(X_{s}, v\right)-P_nG\left( X^{(n)}_{s},v\right)\right\|_{\HH}^2\vartheta (dv)ds\right]^{\frac12} \notag\\
 \le & 4\ee \int_0^t\left\|X_s-X_s^{(n)}\right\|_{\HH}^2ds+4\ee\int_0^t\int_{\mathbb X}\left\|(I-P_n)G(X_s,v)\right\|_{\HH}^2\vartheta (dv)ds\notag\\
 &+4 C_G \ee \left[\int_0^t \sup_{0\le r\le s}\left\|X_r-X^{(n)}_r\right\|_{\HH}^2 \cdot  \left\|X_s-X^{(n)}_s\right\|_{\HH}^2ds\right]^{\frac12}\notag\\
  \le & 4\ee \int_0^t\left\|X_s-X_s^{(n)}\right\|_{\HH}^2ds+4\ee\int_0^t\int_{\mathbb X}\left\|(I-P_n)G(X_s,v)\right\|_{\HH}^2\vartheta (dv)ds\notag\\
 &+\frac{1}{4}\ee \left[ \sup_{0\le s\le t}\left\|X_s-X^{(n)}_s\right\|_{\HH}^2\right]+ 16 C_G^2  \ee \int_0^t\left\|X_s-X^{(n)}_s\right\|_{\HH}^2ds.
\end{align}
For the last term, we have
\begin{align}\label{eq I5}
I_5^{(n)}(t)=&\ee \int_0^t\int_{\mathbb X}\left\|G\left(X_{s-}, v\right)-P_nG\left( X^{(n)}_{s-},v\right)\right\|_{\HH}^2N(ds,dv)\notag\\
=& \ee \int_0^t\int_{\mathbb X}\left\|G\left(X_{s}, v\right)-P_nG\left( X^{(n)}_{s},v\right)\right\|_{\HH}^2\vartheta (dv)ds\notag\\
\le & \ee\int_0^t\int_{\mathbb X}\left\|(I-P_n)G(X_s,v)\right\|_{\HH}^2\vartheta (dv)ds+ C_G \ee \int_0^t\left\|X_s-X^{(n)}_s\right\|_{\HH}^2ds.
\end{align}
Putting  the above inequalities together, we get
\begin{align}\label{eq I6}
& \ee \left[\sup_{0\le s\le t}\left\|X_s-X^{(n)}_s\right\|_{\HH}^2\right]+2\ee \int_0^t\left\|X_s-X^{(n)}_s\right\|_{\VV}^2ds\notag\\
\le &\|(I-P_n)x\|_{\HH}^2+ 2(7+2C_f+17C_{\sigma}^2+17C_G)\ee \int_0^t  \left\|X_s-X^{(n)}_s\right\|_{\HH}^2ds\notag\\
 &+2\ee \int_0^t \left\|(I-P_n)f(X_s)\right\|_{\HH}^2ds
 + 6 \ee  \int_0^t \left\|(I-P_n)\sigma\left(X_s\right) \right\|_{\HS}^2ds\notag\\
&+10\ee\int_0^t\int_{\mathbb X}\left\|(I-P_n)G(X_s,v)\right\|_{\HH}^2\vartheta (dv)ds.
\end{align}

By \eqref{eq: solution-integ2} and \eqref{eq: solution-integ3}, we know that  $\ee \left[\sup_{0\le s\le t}\left\|X_s-X^{(n)}_s\right\|_{\HH}^2\right]<\infty$. Hence, by Gronwall's inequality,  Fatou's lemma and \eqref{eq I6}, we obtain the desired result \eqref{eq convergence 1}.
The proof is complete.
\end{proof}

\section{Transportation cost inequalities for   SPDE with  L\'evy noise and non-Lipschitz reaction term}

\subsection{SPDE with  L\'evy noise and non-Lipschitz reaction term}

Let $\HH, \VV, \beta, \sigma, \widetilde N$ be the same as those in precedent  section. In this section,  we extend  the reaction term $f$ from the Lipschitz case  to   the non-Lipschitz case, for example, one can take $f(x)=-x^3+C_1x$ for some $C_1\in \mathbb R$.

Consider the  following  SPDE on the Hilbert space $\HH$:
\begin{eqnarray}\label{SPDE2}
\left\{
 \begin{array}{lll}
 & dX_t=\Delta X_tdt+f(X_t)dt+\sigma(X_t)d\beta_t+\int_{\mathbb X} G(X_{t-},v)\widetilde{N}(dt,dv); \\
 &X_0=x\in \HH,\\
 \end{array}
\right.
\end{eqnarray}
Suppose that
\begin{itemize}
  \item[(H4)] the reaction term $f$ is a third degree polynomial with the negative leading coefficient,
    \begin{align}f(x)=-x^3+C_1x, \ \ \ \forall x\in \HH, \end{align}
  where $C_1\in\mathbb R$.
  \item[(H5)] $G$ satisfies the following condition:
  \begin{align}
&\int_{\mathbb X}\|G(x, v)\|_{\HH}^6\vartheta(dv)\le C_G'\left(1+\|x\|_{\HH}^6\right).
  \end{align}
\end{itemize}

\begin{definition} An $\HH$-valued right continuous with left limits $(\mathcal F_t)$-adapted process $\{X_t\}_{t\in[0,T]}$ is called a solution of \eqref{SPDE2}, if for its $dt\times \mathbb P$-equivalent class $\hat{X}$, we have
  $\hat X\in \mathbb D([0,T];\HH)\cap L^{2}((0,T];\VV), \ \mathbb P$-a.s. and
  the following equality holds $\mathbb P$-a.s.:
  \begin{align*}
  X_t=&x+ \int_0^t\Delta \bar X_sds+\int_0^tf(\bar X_s)ds+\int_0^t\sigma(\bar X_s)d\beta_s\\
  &+\int_0^t\int_{\mathbb X} G(\bar X_{s-},v)\widetilde{N}(ds,dv),  \ \ t\in[0,T],
  \end{align*}
  where $\bar X$ is any $\VV$-valued progressively measurable $dt\times \mathbb P$ version of $\hat X$.

\end{definition}

Brze\'zniak {\it et al.} \cite{BLZ} proved the following result for the solution of Eq. \eqref{SPDE2}.

\begin{theorem} \cite[Theorem 1.2 and Example 2.2]{BLZ}  Under  (H2)-(H5), for any $x\in L^6(\Omega, \mathcal F_0, \mathbb P; \HH)$, \eqref{SPDE2} admits a unique solution $\{X_t\}_{t\in[0,T]}$, and there exists a constant $C>0$ such that
\begin{align}\label{eq solut integ}
\mathbb E\left(\sup_{t\in[0,T]}\|X_t\|_{\HH}^6\right)+\mathbb E\int_0^T \|X_t\|_{\HH}^4\cdot\|X_t\|_{\VV}^2dt\le C\left(1+\mathbb E \|x\|_{\HH}^6\right).
\end{align}
\end{theorem}

 \begin{theorem}\label{main 2}
 Assume   (H2)-(H5) hold with $2C_1+C_{\sigma}^2+C_G<2\lambda_1$,  $\|\sigma(x)\|_{\rm HS}\le
\bar\sigma$ for any $x\in \HH$ and there is some
Borel-measurable function $\bar G(u)$ on $\mathbb X$
such that $|G(x, u)|\le\bar G(u)$ for all  $x\in \HH, u\in \mathbb X$ and \eqref{condi} holds. Then the statements in Theorem  \ref{main1} hold with $K=2\lambda_1-2C_1-C_{\sigma}^2-C_G$.

\end{theorem}

\begin{remark} The condition $2C_1+C_{\sigma}^2+C_G<2\lambda_1$  guarantees the global dissipation for the system \eqref{SPDE2}, and we could apply finite dimensional SDE's results in \cite{Ma}.  When $C_1$ is large,  the system \eqref{SPDE2} is dissipative outside a bounded set,   by M. Majka's work \cite{Majka}, one expects that the transportation cost inequalities should hold.  However, in  \cite{Majka},    the non-degenerated conditions of  the noises are assumed to make the mirror coupling successful.   Thus, to remove the restriction of $C_1$ in  Theorem \ref{main 2}, we need some extra non-degenerated conditions of the noises.  This is not the task of this paper,  and we hope study it in  future.
\end{remark}

 \begin{proof}[Proof of Theorem \ref{main 2}]
Recall that $P_n$ is the projection mapping from $\VV^*$ into $\HH_n$ defined by \eqref{eq proj}.  For any $n\ge1$, consider the following stochastic differential equation on $\HH_n$:
\begin{eqnarray}\label{eq approx 2}
dX_t^{(n)}&=&P_n \Delta X_t^{(n)}dt+P_nf\left(X_t^{(n)}\right)dt+P_n\sigma\left(X_t^{(n)}\right) d\beta^{(n)}_t\nonumber\\
 & &+\int_{\mathbb{X}}P_n G\left(X_{t-}^{(n)},v\right)\widetilde N(dt,dv),
\end{eqnarray}
with initial condition $X^{(n)}_0=P_nx$.   According to \cite[Theorem 3.1]{ABW}, \eqref{eq approx 2} admits a unique strong solution $X^{(n)}$ satisfying that
\begin{align}\label{eq solut approx22}
   X_t^{(n)}=& P_nx+ \int_0^tP_n\Delta   X_s^{(n)}ds+\int_0^tP_nf\left(X_s^{(n)}\right)ds+\int_0^tP_n\sigma\left(X_s^{(n)}\right)d\beta_s^{(n)}\notag\\
  &+\int_0^t\int_{\mathbb X} P_nG\left(X_{s-}^{(n)},v\right)\widetilde{N}(ds,dv),  \ \ t\in[0,T].
 \end{align}
 Furthermore, using the same method in the proof of \eqref{eq solut integ}, we have
 \begin{align}\label{eq solut integ3}
\sup_{n\ge1} \mathbb E\left(\sup_{t\in[0,T]}\left\|X_t^{(n)}\right\|_{\HH_n}^6+ \int_0^T \left\|X_t^{(n)}\right\|_{\HH_n}^4\cdot\left\|X_t^{(n)}\right\|_{\VV_n}^2dt\right)
\le C\left(1+\mathbb E \|x\|_{\HH}^6\right).
\end{align}

According to \cite[Theorem 2.2]{Ma} and  Lemma \ref{lem:stability:weak},    Theorem \ref{main1} is established
once the following statements are proved:
   \begin{itemize}
  \item[(C1).] $\left\{X^{(n)}\right\}_{n\ge1}$ converges in distribution  to $X$ in $L^2([0,T];  \HH)$ as $n\rightarrow \infty$;
  \item[(C2).] $\left\{X^{(n)}_T\right\}_{n\ge1}$ converges in distribution to $\{X_T\}$ in $\HH$ as $n\rightarrow \infty$.
\end{itemize}
In the sequel, we will prove Conditions   (C1) and  (C2). The proof is complete.
\end{proof}

Let $(\mathbb U,\|\cdot\|_{\mathbb U})$ be a separable  metric space. Given $p>1$, $\alpha\in(0,1)$, let $W^{\alpha,p}([0,T];\mathbb U)$ be the Sobolev space of all $u\in L^p([0,T];\mathbb U)$ such that
$$
\int_0^T\int_0^T\frac{\|u(t)-u(s)\|^p_{\mathbb U}}{|t-s|^{1+\alpha p}}dtds<\infty,
$$
endowed with the norm
$$
\|u\|^p_{W^{\alpha,p}([0,T];\mathbb U)}=\int_0^T\|u(t)\|^p_{\mathbb U}dt+\int_0^T\int_0^T\frac{\|u(t)-u(s)\|^p_{\mathbb U}}{|t-s|^{1+\alpha p}}dtds.
$$

\begin{lemma}\label{Compact}\cite[Sect. 5, Ch. I]{Lions},
  \cite[Sect. 13.3]{Temam 1983}.
Let $\mathbb U\subset \mathbb  Y\subset \mathbb U^*$ be Banach spaces, $\mathbb U$ and $\mathbb U^*$ reflexive, with compact embedding of $\mathbb U$ in $\mathbb Y$.
For any  $p\in(1,\infty)$ and $\alpha\in(0,1)$, let
$
\Gamma=L^p([0,T];\mathbb U)\cap W^{\alpha,p}([0,T];\mathbb U^*)
$
endowed with the natural norm. Then the embedding of $\Gamma$ in $L^p([0,T];\mathbb Y)$ is compact.
\end{lemma}

We first give a priori estimates for $X^{(n)}$.
\begin{lemma}\label{Lemma-estimation}  Under (H2)-(H5), we have

\begin{eqnarray}\label{estimation-01}
\sup_{n\ge1}
 \mathbb{E}\left[\sup_{0\leq t\leq T}\left\| X_t^{(n)}\right\|_{\HH}^{2}
 +\int_0^T\left\|X_t^{(n)}\right\|_{\VV}^{2}dt
 \right]
<\infty,
\end{eqnarray}
and for any $\alpha\in(0,1/2)$,
\begin{eqnarray}\label{estimation-02}
\sup_{n\ge1}\mathbb{E}\left[\left\|X^{(n)}\right\|_{W^{\alpha,2}([0,T],\VV^*)}\right]<\infty.
\end{eqnarray}

\end{lemma}
\begin{proof} Applying  It\^o's formula  with $p=2$ (instead of taking $p=\beta+2$) in the proof of \cite[Lemma 4.2]{BLZ}),  one can obtain the estimate \eqref{estimation-01}. The  details  are omitted here. Next, we  prove \eqref{estimation-02}.
Notice that
\begin{eqnarray}\label{X-epsilon-1}
 X_t^{(n)}
&=&
 P_nx+ \int_{0}^{t}P_n \Delta X_s^{(n)}ds+\int_{0}^{t}P_n f\left(X_s^{(n)}\right)ds\nonumber\\
 & &+
 \int_{0}^{t}P_n\sigma\left(X_s^{(n)}\right)d\beta^{(n)}_s+
 \int_{0}^{t}\int_{\mathbb{U}}P_nG\left(X^{(n)}_{s-},v\right) \widetilde N(ds,dv)\nonumber\\
 &=:& J_1^{(n)}+J_2^{(n)}(t)+ J_3^{(n)}(t)+J_4^{(n)}(t)+J_5^{(n)}(t).
\end{eqnarray}
By the same arguments as in the proof of Theorem 3.1 in \cite{Flandoli-Gatarek}, we have
\begin{eqnarray}\label{1J1-X}
\sup_{n\ge1}\mathbb{E}\left\|J_1^{(n)}\right\|_{\HH}^2<\infty,\ \
\sup_{n\ge1}\mathbb{E}\left\|J_2^{(n)}\right\|^2_{W^{1,2}([0,T];\VV^*)} <\infty.
\end{eqnarray}
Since for $t>s$,
\begin{eqnarray*}
 \mathbb{E}\left\|J_3^{(n)}(t)-J_3^{(n)}(s)\right\|_{\HH}^2
&=&
 \mathbb{E}\left\|\int_s^t P_n f\left(X^{(n)}_r\right)dr\right\|_{\HH}^2\nonumber\\
&\leq& C
 \mathbb{E}\left( \int_s^t\sqrt{1+\left\|X^{(n)}_r\right\|^6_{\HH}} dr\right)^2\nonumber\\
&\leq&
C \mathbb{E}\left(1+\sup_{r\in[0,T]}\left\|X^{(n)}_r\right\|^6_{\HH}\right)(t-s),
\end{eqnarray*}
we have
\begin{eqnarray}\label{1J4-1-X}
 \mathbb{E}\int_0^T\left\|J_3^{(n)}(t)\right\|^2_{\HH}dt
\leq
  C\mathbb{E}\left(1+\sup_{r\in[0,T]}\left\|X^{(n)}_r\right\|^6_{\HH}\right)T^2,
\end{eqnarray}
and
\begin{eqnarray}\label{1J4-2-X}
\mathbb{E}\int_0^T\int_0^T\frac{\left\|J_3^{(n)}(t)-J_3^{(n)}(s)\right\|^2_{\HH}}{|t-s|^{1+2\alpha}}dtds
\leq C(\alpha, T) \mathbb{E}\left(1+\sup_{r\in[0,T]}\left\|X^{(n)}_r\right\|^6_{\HH}\right).
\end{eqnarray}
By \eqref{eq solut integ3}, \eqref{1J4-1-X} and \eqref{1J4-2-X}, we obtain
\begin{eqnarray}\label{1J4-X}
\sup_{n\ge1}\mathbb{E}\left\|J_3^{(n)}\right\|^2_{W^{\alpha,2}([0,T];\VV^*)}<\infty.
\end{eqnarray}
Now for $J_4^{(n)}$, since
for $t>s$,
\begin{eqnarray*}
 \mathbb{E}\left\|J_4^{(n)}(t)-J_4^{(n)}(s)\right\|_{\HH}^2
&=&
 \mathbb{E}\left\|\int_s^t  P_n\sigma\left(X^{(n)}_r\right)d\beta^{(n)}_r\right\|_{\HH}^2\nonumber\\
 &\leq&
 C\mathbb{E}\left(\int_s^t\left\| \sigma\left(X^{(n)}_r\right)\right\|^2_{\HS}dr\right)\nonumber\\
&\leq&
 C C_{\sigma}^2 \mathbb{E}\left(\int_s^t \left(1+\left\|X^{(n)}_r\right\|^2_{\HH}\right)dr\right)\nonumber\\
 &\leq&
C C_{\sigma}^2 \mathbb{E}\left(1+\sup_{r\in[0,T]}\left\|X^{(n)}_r\right\|^2_{\HH}\right)(t-s)\nonumber,
\end{eqnarray*}
similarly  to (\ref{1J4-X}), we have
\begin{eqnarray}\label{1J5-X}
\sup_{n\ge1}\mathbb{E}\left\|J_4^{(n)}\right\|^2_{W^{\alpha,2}([0,T];\VV^*)}<\infty.
\end{eqnarray}
For $J_5^{(n)}$, we also have
\begin{eqnarray*}
 \mathbb{E}\left\|J_5^{(n)}(t)-J_5^{(n)}(s)\right\|_{\HH}^2
&=&
 \mathbb{E}\left\|\int_s^t \int_{\mathbb{X}}G\left(X^{(n)}_{r-},v\right)\widetilde N(dr,dv)\right\|_{\HH}^2\\
&\le& C
 \mathbb{E}\int_s^t \int_{\mathbb{X}}\left\|G\left(X^{(n)}_r,v\right)\right\|_{\HH}^2\vartheta(dv)dr\\
&\leq&  C C_G \mathbb{E}\left(1+\sup_{r\in[0,T]}\left\|X^{(n)}_r\right\|^2_{\HH}\right)(t-s).
\end{eqnarray*}
Similarly to (\ref{1J4-X}), we have
\begin{eqnarray}\label{1J6-X}
\sup_{n\ge1}\mathbb{E}\left\|J_5^{(n)}\right\|^2_{W^{\alpha,2}([0,T];\VV^*)}<\infty.
\end{eqnarray}
Putting above inequalities together, we  get (\ref{estimation-02}).
 The proof is complete.
\end{proof}

\begin{proposition}\label{prop:approx}  Under  (H2)-(H5), for any $T>0$,
\begin{itemize}
  \item[(a).] $\left\{X^{(n)}\right\}_{n\ge1}$ converges in distribution  to $X$ in $L^2([0,T];  \HH)$ as $n\rightarrow \infty$;
  \item[(b).] $\left\{X^{(n)}_T\right\}_{n\ge1}$ converges in distribution to $\{X_T \}$ in $\HH$ as $n\rightarrow \infty$.
\end{itemize}
\end{proposition}

\begin{proof}
 (a). For any subsequence $\{X^{(n_k)}\}_{k\ge1}\subset \{X^{(n)}\}_{n\ge1}$,
 by Lemma \ref{Compact} and  Lemma \ref{Lemma-estimation},   we know that  $\{X^{(n_k)}\}_{k\ge1}$ is tight in the space $L^2([0,T]; \HH)$. Hence, there exists a subsequence $\{X^{(n'_k)}\}_{k\ge1}\subset \{X^{(n_k)}\}_{k\ge1}$,  which converges in distribution   as random variables in the space $L^2([0,T]; \HH)$.
By  the uniqueness of the limit (see the proof of Theorem 4.1 in \cite{BLZ}) and the arbitrariness of the  subsequence $\left\{X^{(n_k)}\right\}_{k\ge1}$,  we know that $\left\{X^{(n)}\right\}_{n\ge1}$ converges in distribution  to $X$ in $L^2([0,T];  \HH)$ as $n\rightarrow \infty$.

(b). Recall that $\{S(t)\}_{t\ge0}$ is the analytic semigroup associated with $\Delta$. Let  $S^{(n)}(t)=P_nS(t)$. According to  \cite[Chapter 9.3]{PZ}, the  solution  $\{X_t\}_{t\ge0}$ to  \eqref{SPDE2} is equivalent to the following form:
\begin{align}\label{eq:mild solution1}
 X_t=&S(t)x+\int_0^tS(t-s)f(X_s)ds+\int_0^tS(t-s)\sigma(X_s)d\beta_s\notag\\
 &+\int_0^t\int_{\mathbb X} S(t-s)G(X_{s-},v)\widetilde N(ds,dv),
\end{align}
and the solution $\left\{X^{(n)}_t\right\}_{t\ge0}$ to   \eqref{eq approx 2} is  equivalent to following form:
\begin{align}\label{eq:mild solution2}
 X^{(n)}_t=&S^{(n)}(t)P_nx+\int_0^tS^{(n)}(t-s)f\left(X^{(n)}_s\right)ds+\int_0^tS^{(n)}(t-s)\sigma\left(X^{(n)}_s\right)P_nd\beta_s\notag\\
 &+\int_0^t\int_{\mathbb X} S^{(n)}(t-s)P_nG\left(X^{(n)}_{s-},v\right)\widetilde N(ds,dv).
\end{align}

 Applying  a generalized version of the Skorokhod representation theorem (e.g., see \cite[Theorem C.1]{BHR}), there exist a stochastic basis $(\bar {\Omega},  \bar{\mathcal F}, (\bar{\mathcal F_t})_{t\ge0}, \bar {\pp})$ and  the random variables $$\left\{\left(\bar x^{(n)}, {\bar X^{(n)}}, \bar x, \bar {X}, \bar {\beta}, \bar {N}\right)\right\}_{n\ge1}$$  on this basis satisfying  that $\left(\bar x^{(n)},\bar {X}^{(n)}, \bar x, \bar {X}, \bar {\beta}, \bar {N}\right)$  has the same law as $\left(x^{(n)}, X^{(n)}, x, X, \beta, N\right)$ for any $n\ge1$, $\bar x^{(n)}\rightarrow \bar x$ in $\HH$, $\bar {\mathbb P}$-a.s.,  and  $\bar {X}^{(n)}\rightarrow \bar {X}$ in $L^2([0,T]; \HH)$, $\bar {\mathbb P}$-a.s.. Next,  we prove that $\bar {X}^{(n)}_T$ converges to $\bar {X}_T$  in probability under $\bar {\mathbb P}$.

For any $n\ge1, M>0$, let
 \begin{align}\label{eq Omega 1}
\bar {\Omega}_{n,M}=\left\{\bar {\omega};\   \sup_{t\in[0,T]} \left\|\bar {X}^{(n)}_t(\bar {\omega})\right\|_{\HH}\vee \left\|\bar {X}_t(\bar{\omega})\right\|_{\HH} \le M\right\}.
\end{align}
Then by \eqref{eq solut integ}, \eqref{estimation-01} and Fatou's lemma, we know that
 \begin{align}\label{eq Omega 2}
\lim_{M\rightarrow\infty}\sup_{n\ge1} \bar {\mathbb P}\left(\bar {\Omega}_{n,M}^c\right)=0,
\end{align}
and for any $M>0$, by the dominated  convergence theorem, we have
 \begin{align}\label{eq Omega 3}
\lim_{n\rightarrow\infty} \mathbb E^{\bar {\mathbb P}} \left(\int_0^T \left\|\bar X_t-\bar X_t^{(n)} \right\|_{\HH}^2dt\cdot 1_{\bar {\Omega}_{n,M}}\right)
=0.
\end{align}
Next, we will prove that
 \begin{align}\label{eq Omega 4}
\lim_{n\rightarrow\infty} \mathbb E^{\bar {\mathbb P}}  \left( \left\|\bar X_T-\bar X_T^{(n)} \right\|_{\HH}^2\cdot 1_{\bar {\Omega}_{n,M}}\right)=0.
\end{align}
This, together with \eqref{eq Omega 2}, implies (b).

By \eqref{eq:mild solution1}  and \eqref{eq:mild solution2}, we have for any $t\in[0,T]$,
\begin{align}\label{eq:mild solution3}
&\left\|\bar X_t-\bar X^{(n)}_t\right\|_{\HH}\notag\\
\le & \left\|S(t)\bar x- S^{(n)}(t)P_n\bar x\right\|_{\HH}\notag\\
&+\left\|\int_0^t\left[S(t-s)f(\bar X_s)-S^{(n)}(t-s)f\left(\bar X^{(n)}_s\right)\right]ds\right\|_{\HH}\notag\\
&+\left\|\int_0^t\left[S(t-s)\sigma\left(\bar X_s\right)-S^{(n)}(t-s)\sigma\left(\bar X^{(n)}_s\right)\right]d\beta_s\right\|_{\HH}\notag\\
 &+\left\|\int_0^t\int_{\mathbb X} \left[S(t-s)G(\bar X_{s-},v)- S^{(n)}(t-s)P_nG\left(\bar X^{(n)}_{s-},v\right)\right]\widetilde {\bar N}(ds,dv)\right\|_{\HH}\notag\\
 =:& J_{1,n}(t)+J_{2,n}(t)+J_{3,n}(t)+J_{4,n}(t).
\end{align}
By the dominated  convergence theorem, we can prove that  for $k=1,\cdots, 4, t\in[0,T]$,
\begin{align}\label{eq Omega 4}
\lim_{n\rightarrow\infty} \mathbb E^{\bar {\mathbb P}}\left[ J_{k,n}(t)\cdot 1_{\bar {\Omega}_{n,M}}\right]=0.
\end{align}
Here,  we will only prove \eqref{eq Omega 4} for $k=2$ and the other term can be proved similarly but more easily.
 Notice that
\begin{align}\label{eq:mild solution3}
 \mathbb E^{\bar{\mathbb P}}\left[ J_{2,n}(t)\cdot 1_{\bar {\Omega}_{n,M}}\right]
\le & \mathbb E^{\bar{\mathbb P}}\left\|\int_0^t(I-P_n)S(t-s) f\left(\bar X_s\right)ds\cdot 1_{\bar {\Omega}_{n,M}}\right\|_{\HH}\notag\\
&+ \mathbb E^{\bar{\mathbb P}}\left\|\int_0^tS^{(n)}(t-s)\left[f(\bar X_s)-f\left(\bar X^{(n)}_s\right)\right]ds\cdot 1_{\bar {\Omega}_{n,M}}\right\|_{\HH}.
 \end{align}
By (2.10) in \cite{Xu} and the  Sobolev embedding theorem, we have
 \begin{align}
\|f(x)\|_{\HH}\le C\left(1+\|x\|_{\HH_{1/6}}^3\right)\le C(1+\|x\|_{\HH}^2\cdot \|x\|_{\VV}).
\end{align}
Since $P_n\rightarrow I$ as $n\rightarrow\infty$, by \eqref{eq solut integ} and the dominated  convergence theorem, we have
 \begin{align}
 & \mathbb E^{\bar{\mathbb P}}\int_0^t\left\|(I-P_n)S(t-s)f\left(\bar X_s\right)\right\|_{\HH}ds\notag\\
 \le & C \mathbb E^{\bar{\mathbb P}}\int_0^t\left\|(I-P_n)\right\|\cdot   \left(1+\|\bar X_s\|_{\HH}^2\cdot \|\bar X_s\|_{\VV}\right)ds
  \longrightarrow 0, \ \ \ \text{as}\ \ \  n\rightarrow\infty.
 \end{align}

By (2.8) in \cite{Xu} and the  Sobolev embedding theorem, we have
 \begin{align}
\|f(x)-f(y)\|_{\HH}\le & C \left(1+\|x\|_{\HH_{1/4}}^2+\|y\|_{\HH_{1/4}}^2 \right) \|x-y\|_{\HH}\notag\\
\le & C\left(1+\|x\|_{\HH}\cdot\|x\|_{\VV}+\|y\|_{\HH}\cdot\|x\|_{\VV} \right) \|x-y\|_{\HH}.
\end{align}
Then,  by \eqref{eq solut integ}, \eqref{eq solut integ3} and \eqref{eq Omega 3}, we have
 \begin{align}
&\mathbb E\left\|\int_0^tS^{(n)}(t-s)\left[f(\bar X_s)-f\left(\bar X^{(n)}_s\right)\right]ds\cdot 1_{\bar {\Omega}_{n,M}}\right\|_{\HH}\notag\\
\le & C \mathbb E\left(\int_0^t \left(1+\left\|\bar X_s\right\|_\HH\cdot\|\bar X_s\|_{\VV}+\|\bar X_s^{(n)}\|_\HH\cdot\|\bar X_s^{(n)}\|_{\VV} \right) \|\bar X_s-\bar X_s^{(n)}\|_{\HH}ds\cdot 1_{\bar {\Omega}_{n,M}}\right)\notag\\
\le & C  \left[\mathbb E\int_0^t \left(1+\left\|\bar X_s\right\|_\HH\cdot\|\bar X_s\|_{\VV}+\|\bar X_s^{(n)}\|_\HH\cdot \|\bar X_s^{(n)}\|_{\VV} \right)^2\cdot 1_{\bar {\Omega}_{n,M}}ds\right]^{\frac12}\notag\\
&\ \ \cdot  \left[\mathbb E\int_0^t\left\|\bar X_s-\bar X_s^{(n)}\right\|_{\HH}^2ds\cdot 1_{\bar {\Omega}_{n,M}}\right]^{\frac12}\notag\\
&\longrightarrow 0, \ \ \text{as } n\rightarrow\infty.
 \end{align}

The proof is complete.
\end{proof}\

\noindent{\bf Acknowledgments} We sincerely thank the references  for  helpful comments and remarks.

\end{document}